\documentclass[reqno,12pt]{amsart}

\usepackage{amsmath,amsfonts,amsthm,amssymb}

\newtheorem{theorem}{Theorem}[section]
\newtheorem{lemma}[theorem]{Lemma}

\newtheorem{proposition}[theorem]{Proposition}

\theoremstyle{definition}

\newtheorem{example}[theorem]{Example}

\newtheorem{remark}[theorem]{Remark}

\newcommand{\ord}{\operatorname{ord}}
\newcommand{\lcm}{\operatorname{lcm}}

\newcommand{\fix}{\mathsf F}
\newcommand{\orbit}{\mathsf O}

\newcommand{\dirichlet}{\delta}

\def\bigo{\operatorname{O}}
\def\({\left(}
\def\){\right)}
\def\dirichlet{\mathsf{d}}
\def\imag{{\rm i}}
\renewcommand{\rho}{\varrho}
\renewcommand{\le}{\leqslant}
\renewcommand{\ge}{\geqslant}

\newcommand{\divides}{|}

\DeclareMathOperator{\residue}{Res}

\title[]{Orbits for products of maps}
\author{Apisit Pakapongpun and Thomas Ward}
\date{Draft - \today}

\address{School of Mathematics, University of East Anglia,
Norwich, NR4 7TJ, UK}
\email{t.ward@uea.ac.uk}

\subjclass[2010]{37P35}

\begin{document}

\begin{abstract}
We study the behaviour of the dynamical
zeta function and the orbit Dirichlet series
for products of maps. The behaviour under
products of the
radius of convergence
for the zeta function, and the abscissa of
convergence for the orbit Dirichlet series,
are discussed. The
orbit Dirichlet series of the cartesian
cube of a map with one orbit
of each length is shown to have a natural boundary.
\end{abstract}

\maketitle

\section{Introduction}

A fundamental topological invariant of a dynamical system --
here thought of as a continuous map~$T:X\to X$ of a compact
metric space -- is its orbit-counting data. Analytic properties
of functions capturing this data have been widely exploited in
dynamics. Recently the authors~\cite{MR2486259} studied
functorial properties of orbit-counting functions, relating
disjoint unions, Cartesian products, and iterates of maps to
corresponding operations on the orbit-counting functions. Here
we focus on some analytic questions in the same spirit, a
simple example being this: What is the relationship between the
analytic properties of the dynamical zeta
functions~$\zeta_{T_1}$,~$\zeta_{T_2}$ and~$\zeta_{T_1\times
T_2}$? Similar questions arise for the orbit Dirichlet series
introduced in~\cite{emsw}, where analytic properties are
directly related by Tauberian theorems to the usual
orbit-growth function~$\pi_T$.

In order to highlight the underlying combinatorial questions,
we take a cavalier attitude to maps in the following sense. For
any sequence~$a=(a_n)_{n\ge1}$ of non-negative integers, there
is (manifestly) a map on~$\mathbb N$ with~$a_n$ closed orbits
of length~$n$ for each~$n\ge1$; via a compactification there is
a continuous map on a compact metric space with the same
property; finally, via a beautiful theorem of
Windsor~\cite{MR2422026}, there is a~$C^{\infty}$
diffeomorphism of the two-torus with the same property. Thus
all our remarks below may be seen as being about abstract
combinatorial maps or about (unspecified) smooth examples. In
the former setting, the paradigmatic examples are those for
which the sequence~$a$ is arithmetically simple, the prototype
being a map with exactly one orbit of each length, with the
natural analytic tool being the orbit Dirichlet series. In the
latter setting, the paradigmatic example might be an Axiom~A
diffeomorphism of the torus, with the natural tool being the
dynamical zeta function. Thus two examples of the arguments
below are the following. Firstly, if~$T$ has one orbit of each
length, then the orbit Dirichlet series~${\mathsf d}_T$ is the
Riemann zeta function, and a calculation shows that~${\mathsf
d}_{T\times T}(s)=\frac{\zeta(s)^2\zeta(s-1)}{\zeta(2s)}$, with abscissa
of convergence at~$2$ and a meromorphic extension to the plane;
more surprising is the fact that for the Cartesian cube we
find that~${\mathsf d}_{T\times T\times T}(s)$ has abscissa of
convergence at~$3$, a meromorphic extension to~$\Re(s)>1$, and
a natural boundary at~$\Re(s)=1$. This is a striking instance
of a naturally-occurring Dirichlet series with a natural
boundary. Secondly, if~$T_1$ and~$T_2$ are maps with rational
dynamical zeta functions, what relates the discs of
convergence of~$\zeta_{T_1}$ and $\zeta_{T_2}$ to that
of~$\zeta_{T_1\times T_2}$?

\section{Products and Iterates}

Let~$T$ (or~$T_1,T_2,\dots$) be maps. A closed orbit~$\tau$ of
length~$\vert\tau\vert$ is a set of the
form~$\{x,Tx,\dots,T^{\vert\tau\vert}x=x\}$ with
cardinality~$\vert\tau\vert$; write~$\orbit_T(n)$ for the
number of closed orbits of length~$n$ under~$T$. We always
assume that~$\orbit_T(n)<\infty$ for all~$n\ge1$.

The number of
points of period~$n$ (that is, the number of points fixed by
the iterate~$T^n$) is~$\fix_T(n)=\sum_{d\vert n}d\orbit_T(d)$.
The dynamical zeta function associated to~$T$ is the function
\[
\zeta_T(z)=\exp\sum_{n=1}^{\infty}\fix_T(n)\frac{z^n}{n},
\]
with radius of
convergence~$\rho(\zeta_T)=1/\limsup_{n\to\infty}\fix_T(n)^{1/n}$
(which may be zero),
and the orbit Dirichlet series associated to~$T$ is
\[
{\mathsf d}_T(s)=\sum_{n=1}^{\infty}\frac{\orbit_T(n)}{n^s},
\]
convergent on a (possibly empty) half-plane~$\Re(s)>\sigma(\dirichlet_T)$,
where~$\sigma(\dirichlet_T)$ is the abscissa of convergence.
Analytic properties of~$\zeta_{T}$ and~$\dirichlet_T$ may be
used in several ways, the most immediate being that asymptotics for
\[
\pi_T(N)=\vert\{\tau\mid\vert\tau\vert\le N\}\vert
\]
may be found via Tauberian theorems.
The usual
M{\"o}bius relation between the sequences~$(\orbit_T(n))$
and~$(\fix_T(n))$ means that
\begin{equation}\label{andtheycamefromeverywhere}
{\mathsf d}_T(s)=\frac{1}{\zeta(s+1)}\sum_{n=1}^{\infty}\frac{\fix_T(n)}{n^{s+1}},
\end{equation}
and, viewed via the Euler transform, the same relation means
that
\[
\zeta_T(z)=
\prod_{\tau}(1-z^{\vert\tau\vert})^{-1}=
\prod_{n=1}^{\infty}(1-z)^{-\orbit_T(n)}.
\]
Clearly~$\fix_{T_1\times T_2}(n)=\fix_{T_1}(n)\fix_{T_2}(n)$
for all~$n\ge1$, and
as pointed out in~\cite[Lem.~1]{MR2486259},
it follows that
\begin{equation}\label{shecamefromprovidence}
\orbit_{T_1\times T_2}(n)=\sum_{\lcm(d_1,d_2)=n}
\gcd(d_1,d_2)\orbit_{T_1}(d_1)\orbit_{T_2}(d_2)
\end{equation}
(this may be seen using \eqref{andtheycamefromeverywhere}
or by pure thought). The arithmetic properties of
the operation~\eqref{shecamefromprovidence} are rather
subtle.

Turning now to iterates of a single map (rather than
products of pairs of maps),
write~$\mathcal D(n)$ for the set of prime divisors
of~$n\in\mathbb N$, and for a prime decomposition~$n=
\boldsymbol p^{\boldsymbol a}=p_1^{a_1}\cdots p_r^{a_r}$
and a subset~$J\subset\mathcal D(n)$,
write~$\boldsymbol p_J^{\boldsymbol a_J}$
for the restricted product~$\prod_{p_j\in J}p_j^{a_j}$.
The basic formula for orbit-counting under
iteration is found in~\cite[Th.~4]{MR2486259}:
if~$m=\boldsymbol p^{\boldsymbol a}$ and~$J=J(n)=
\mathcal D(m)\setminus\mathcal D(n)$, then
\begin{equation}\label{theoneinrhodeisland}
\orbit_{T^m}(n)=\sum_{d\divides\boldsymbol p_J^{\boldsymbol a_J}}
\textstyle\frac{m}{d}\orbit_{T}(\frac{mn}{d}).
\end{equation}
In this expression~$J$ depends on~$n$, so it involves a
splitting into cases depending on the set of primes
dividing~$n$. The corresponding formula for fixed points is
once again
trivial:~$\fix_{T^k}(n)=\fix_{T}(kn)$ for all~$n,k\ge1$.

\begin{example}\label{wheretheoldworldshadowshang}
The quadratic map~$T:x\mapsto1-cx^2$ on the
interval~$[-1,1]$ at the Feigenbaum
value~$c=1.401155\cdots$ has exactly one
orbit of length~$2^k$ for each~$k\ge0$, so
(as pointed out by Ruelle~\cite{ruelle})
\[
\zeta_T(z)=\prod_{n=0}^{\infty}\(1-z^{2^n}\)^{-1}=\prod_{n=0}^{\infty}
\(1+z^{2^n}\)^{n+1},
\]
satisfying the functional equation~$\zeta_T(z^2)=(1-z)\zeta_T(z)$.
More enlightening from an analytic point of view is to
note that
\begin{equation}\label{earlyin2010}
\dirichlet_T(s)=\frac{1}{1-2^{-s}},
\end{equation}
with~$\sigma(\dirichlet_T)=0$.
It is clear that~$\pi_T(N)=\frac{\log N}{\log 2}+\bigo(1)$;
this toy case may also be found by applying Perron's theorem~\cite{perron}
or Agmon's Tauberian theorem~\cite{MR0054079}
to~\eqref{earlyin2010}.
Even in this simple case
some care is needed as there are infinitely many poles
on the critical line~$\Re(s)=0$, and the corresponding
residue sums are only conditionally convergent.
A calculation
using~\eqref{theoneinrhodeisland}
(see~\cite{MR2486259} for the
details)
shows that
\[
\dirichlet_{T^k}(s)=\vert k\vert_2^{-1}-1+\vert k\vert_2^{-1}\dirichlet_T(s),
\]
so in this case~$\sigma(\dirichlet_{T^k})=\sigma(\dirichlet_T)$ for all~$k\ge1$.
Similarly,
\[
\dirichlet_{T\times T}(s)=\frac{3}{1-2^{-(s-1)}}-\frac{2}{1-2^{-s}},
\]
so~$\sigma(\dirichlet_{T\times T})=
\sigma(\dirichlet_T)+\sigma(\dirichlet_T)+1$ in this case.
\end{example}

In pursuit of the behaviour of the abscissa of convergence
for products, Ramanujan's formula~\cite{45.1250.01} for
the Dirichlet series
with coefficients~$\sigma_a(n)\sigma_b(n)$ may be used together
with~\eqref{shecamefromprovidence} to give the following
(the detailed calculation is in the first author's
thesis~\cite{pakapongpun}).

\begin{example}\label{heavyintheair}
Let~$T_1$ be map with~$n^a$ orbits of length~$n$ and
let~$T_2$ be a map
with~$n^b$ orbits of length~$n$, so that~$\dirichlet_{T_1}(s)=\zeta(s-a)$
and~$\dirichlet_{T_2}(s)=\zeta(s-b)$. Then
\[
\dirichlet_{T_1\times T_2}(s)=\frac{\zeta(s-a)\zeta(s-b)\zeta(s-a-b-1)}{\zeta(2s-a-b)}.
\]
Thus~$\sigma(\dirichlet_{T_1\times T_2})=\sigma(\dirichlet_{T_1})+
\sigma(\dirichlet_{T_2})$ in this case.
Perron's theorem applies to show that
\begin{eqnarray*}
\pi_{T_1\times T_2}(N)&\sim&
\residue\(\dirichlet_{T_1\times T_2}(s)N^s/s\)_{s=a+b+2}\\
&=&
\textstyle\frac{\zeta(a+2)\zeta(b+2)}{2\zeta(a+b+4)+(a+b)\zeta(a+b+4)}N^{a+b+2}.
\end{eqnarray*}
\end{example}

\begin{example}\label{shepackedherhopesanddreams}
Let~$T_1$ be the full shift on~$a$ symbols, and~$T_2$ the full
shift on~$b$ symbols, so that~$\zeta_{T_1}(z)=1/(1-az)$
and~$\zeta_{T_2}(z)=1/(1-bz)$. Clearly in this
case~$\zeta_{T_1\times T_2}(z)=1/(1-abz)$,
so~$\rho(\zeta_{T_1\times
T_2})=\rho(\zeta_{T_1})\rho(\zeta_{T_2})$.
\end{example}

Our first result is that the phenomena in
Example~\ref{shepackedherhopesanddreams} holds for rational
zeta functions. Recall that a linear recurrence sequence is
said to be non-degenerate if among the non-trivial ratios of
zeros of the characteristic polynomial no unit roots are found
(see~\cite[Sect.~1.1.9]{MR1990179}), and we say that a rational
zeta function~$\zeta_T$ is non-degenerate if the linear
recurrence sequence satisfied by the sequence~$\(\fix_T(n)\)$
is non-degenerate.

\begin{theorem}\label{likearefugee}
If~$\zeta_{T_1}$ and~$\zeta_{T_2}$ are non-degenerate rational
functions, then~$\rho(\zeta_{T_1^k})=\rho(\zeta_{T_1})^k$,
and~$\rho(\zeta_{T_1\times
T_2})=\rho(\zeta_{T_1})\rho(\zeta_{T_2})$.
\end{theorem}

\begin{proof}
The first assertion is immediate: if~$\zeta_{T_1}$ is rational,
then by~\cite{MR0271401} there are algebraic
numbers~$\beta_1,\dots,\beta_r$ and~$\alpha_1,\dots,\alpha_s$
with
\begin{equation}\label{sheheardaboutaplacepeopleweresmiling}
\fix_{T_1}(n)=\sum_{i=1}^{r}\beta_i^n-\sum_{i=1}^{s}\alpha_i^n,
\end{equation}
giving the statement at once.

The second statement is more delicate. If
\[
\zeta_{T_j}(s)=\prod_{i=1}^{r^{(j)}}\(1-\alpha_i^{(j)}z\)
\prod_{i=1}^{s^{(j)}}\(1-\beta_i^{(j)}z\)^{-1}
\]
for~$j=1,2$ then
\[
\zeta_{T_1\times T_2}(z)=\frac{\displaystyle
\prod_{i=1}^{r^{(1)}}\prod_{j=1}^{s^{(2)}}\(1-\alpha_i^{(1)}\beta_j^{(2)}z\)
\prod_{i=1}^{r^{(2)}}\prod_{j=1}^{s^{(1)}}\(1-\alpha_i^{(2)}\beta_j^{(1)}z\)
}
{\displaystyle
\prod_{i=1}^{r^{(1)}}\prod_{j=1}^{r^{(2)}}\(1-\alpha_i^{(1)}\alpha_j^{(2)}z\)
\prod_{i=1}^{s^{(2)}}\prod_{j=1}^{s^{(1)}}\(1-\beta_i^{(1)}\beta_j^{(2)}z\)
}.
\]
Thus~$\rho(\zeta_{T_1\times T_2})$ is the reciprocal of
\begin{equation}\label{justasherfathercameacrossthesea}
{\max\{\alpha_i^{(1)}\alpha_j^{(2)},\beta_k^{(1)}\beta_{\ell}^{(2)}\mid
1\negthinspace\le\negthinspace
i\negthinspace\le\negthinspace
r^{(1)},1\negthinspace\le\negthinspace
j\le r^{(2)},1\negthinspace\le\negthinspace k
\negthinspace\le\negthinspace
s^{(1)},
1\negthinspace\le\negthinspace\ell\le s^{(2)}\}},
\end{equation}
and we claim that
the reciprocal of~\eqref{justasherfathercameacrossthesea} is
equal to
\[
\max\{\beta_i^{(1)}\beta_j^{(2)}\mid1\le i\le
s^{(1)},1\le j\le s^{(2)}\}^{-1}.
\]
That is, the exponential
growth due to the poles of the zeta function dominates the
growth due to the zeros. In simple cases like
Example~\ref{shepackedherhopesanddreams} this is obvious, but
in general account needs to be taken of possible cancellation
among terms of equal modulus
in~\eqref{sheheardaboutaplacepeopleweresmiling}.

\begin{lemma}\label{andhowtheylovedtheland}
If~$\zeta_T$ is a non-degenerate zeta function
with~\eqref{sheheardaboutaplacepeopleweresmiling},
then
\[
\max\{\vert\beta_i\vert\mid1\le i\le r\}\ge
\max\{\vert\alpha_i\vert\mid1\le i\le s\}.
\]
\end{lemma}

\begin{proof}
If~$\max\{\vert\alpha_i\vert,\vert\beta_j\vert\}<1$
then~$\fix_T(n)\to0$ as~$n\to\infty$, so~$\fix_T(n)=0$ for all
large~$n$, and therefore the function is degenerate
(see~\cite[Th.~2.1]{MR1990179}). It follows
that~$\max\{\vert\alpha_i\vert,\vert\beta_j\vert\}\ge1$.
If~$\max\{\vert\alpha_i\vert\}=1$,
then~$\max\{\vert\beta_j\vert\}\ge1$ since~$\fix_T(n)\ge0$ for
all~$n\ge1$ and we are done. Assume therefore
that~$\max\{\vert\alpha_i\vert\}>1$, and for the purposes of a
contradiction assume that
\[
1\le\max\{\vert\beta_j\vert\}<\max\{\vert\alpha_i\vert\},
\]
and choose~$\epsilon>0$ so that
\begin{equation}\label{theyspokeabouttheredmansways}
\max\{\vert\beta_j\vert\}<\(\max\{\vert\alpha_i\vert\}\)^{1-\epsilon}.
\end{equation}
By~\cite[Prop.~1]{MR0271401} the numbers~$\alpha_i$
and~$\beta_j$ are algebraic numbers (indeed, are reciprocals of
algebraic integers), so that the estimates of
Evertse~\cite{MR766298} or van der Poorten and
Schlickewei~\cite{MR1119694} may be applied to see that there
is an~$N(T,\epsilon)$ with
\[
\left\vert
\sum_{i=1}^{r}\alpha_i^n\right\vert\ge s\(\max\{\vert\alpha_i\vert\}\)^{n(1-\epsilon)}
\]
for~$n\ge N(T,\epsilon)$. Then,
by~\eqref{theyspokeabouttheredmansways},
\begin{eqnarray*}
\left\vert
\sum_{i=1}^{r}\alpha_i^n\right\vert&>&s\max\{\vert\beta_j\vert\}^n\quad\mbox{(for all large~$n$)}\\
&\ge&\sum_{j=1}^{s}\vert\beta_j\vert^n\quad\mbox{(for all large~$n$)},
\end{eqnarray*}
which would make~$\fix_{T}(n)$ negative for large~$n$,
an impossibility.
\end{proof}
This completes the proof, since
Lemma~\ref{andhowtheylovedtheland} shows
that~$\rho(\zeta_{T_j})=\max\{\vert\beta_i^{(j)}\vert\}$
for~$j=1,2$ and that~$\rho(\zeta_{T_1\times T_2})$ is the
product.
\end{proof}

The next two examples show that the
relationships found in Theorem~\ref{likearefugee}
do not hold in general.

\begin{example}\label{downinthecrowdedbars}
Let~$\mathcal P=\{p_1,p_2,\dots\}$ be the set of primes
written in order,
and let~$\mathcal P_1=\{p_2,p_4,\dots\}$,~$\mathcal P_2=
\{p_1,p_3,\dots\}$ be the primes of even and of odd
index respectively. Let~$T_j$ be a map with
\[
\orbit_{T_j}(n)=\begin{cases}A_j^n&\mbox{if }n\in\mathcal P_j;\\
0&\mbox{if not},\end{cases}
\]
for~$j=1,2$, where~$A_1=2$ and~$A_2=3$.
Then~$\fix_{T_j}(n)=\sum_{p\in\mathcal P_j;\vert n\vert_p<1}pA_j^p$,
and so
\begin{equation}\label{tothegreatdivide}
\fix_{p_{2k+j}}(T_j)^{1/p_{2k+j}}\rightarrow A_j
\end{equation}
as~$k\to\infty$ for each~$j=1,2$.
On the other hand, a simple induction argument shows that
\[
\(a_1A_j^{a_1}+a_2A_j^{a_2}+\cdots+a_rA_j^{a_r}\)^{1/a_1a_2\cdots a_r}
\le A_j
\]
for distinct~$a_1,\dots,a_r\ge1$, so~$A_j$ is in fact
the upper limit in~\eqref{tothegreatdivide},
and~$\rho(\zeta_{T_j})=1/A_j$ for~$j=1,2$.
Turning to the product, let~$n=n_1n_2$, where
\[
n_j=\prod_{i=1}^{u(j)}
q_{i,j}^{a_{i,j}}
\]
with~$q_{i,j}\in\mathcal P_j$ and~$a_{i,j}>0$.
Then
\[
\fix_{T_1\times T_2}(n)=
\(
\sum_{i=1}^{u(1)}s_{i,1}2^{s_{i,1}}
\)
\(
\sum_{i=1}^{u(1)}s_{i,2}3^{s_{i,2}}
\),
\]
and straightforward estimates show that
\[
\limsup_{n\to\infty}\fix_{T_1\times T_2}(n)^{1/n}<6.
\]
Thus, for this example,~$\rho(\zeta_{T_1\times T_2})<
\rho(\zeta_{T_1})\rho(\zeta_{T_2})$.
\end{example}

\begin{example}
The map~$T_1$ from Example~\ref{downinthecrowdedbars}
has~$2^n$ orbits of length~$n$ if~$n\in\mathcal P_1$,
and none otherwise, and we
have seen
in~\eqref{tothegreatdivide}
that~$\rho(\zeta_{T_1})=\frac{1}{2}$.
On the other hand,~$\fix_{T_1^2}(n)
=\fix_{T_1}(2n)=\sum_{p\in\mathcal P_1;\vert n
\vert_p<1}p2^p$, so~$\rho(\zeta_{T_1^2})=\frac12$ also.
\end{example}

Example~\ref{wheretheoldworldshadowshang}
has~$\sigma(\dirichlet_{T_1})+\sigma(\dirichlet_{T_2})-\sigma(\dirichlet_{T_1\times
T_2})=1$; some simple estimates show that this discrepancy
cannot be any larger.

\begin{proposition}
$\sigma(\dirichlet_{T_1\times
T_2})\le\sigma(\dirichlet_{T_1})+\sigma(\dirichlet_{T_2})+1.$
\end{proposition}

\begin{proof}
Let~$\sigma_j=\sigma(T_j)$ for~$j=1,2$. Then, for
any~$\epsilon>0$,~$\dirichlet_{T_j}(\sigma_j+\epsilon)<\infty$
and so
\[
\sum_{n=1}^{\infty}\frac{\fix_{T_j}(n)}{n^{1+\sigma_j+\epsilon}}<\infty
\]
for~$j=1,2$ by~\eqref{andtheycamefromeverywhere}. Thus
\[
\sum_{n=1}^{\infty}\frac{\fix_{T_1\times T_2}(n)}{n^{2+\sigma_1+\sigma_2+2\epsilon}}<\infty,
\]
and therefore~$\dirichlet_{T_1\times
T_2}(1+\sigma_1+\sigma_2+2\epsilon)<\infty$
by~\eqref{andtheycamefromeverywhere} again.
\end{proof}

\section{Higher products}

Even in the simplest of situations, higher products have quite
subtle combinatorial and analytic properties, and for
simplicity we restrict attention to the case of a map with a
single orbit of each length. Similar methods will apply to maps
for which the sequence~$\(\orbit_T(n)\)$ is multiplicative.

\begin{proposition}\label{seekingaplacetostand}
Let~$T$ be a map with~$\orbit_T(n)=n^a$ for~$n\ge1$.
Then~$\orbit_{T\times\cdots\times T}(n)$ is equal to
\[
\prod_{p\divides n}\negthinspace\frac{1}{(p^{a+1}-1)^{m-1}}
\negthinspace\(\negthinspace
\sum_{r=0}^{m-1}(-1)^r\binom{m}{m-r}
p^{((m-r)(a+1)-1)\ord_p(n)}
\negthinspace\negthinspace\sum_{j=0}^{(m-r)(a+1)-1}
\negthinspace\negthinspace p^j\negthinspace\)
\]
where there are~$m$ terms in the Cartesian product.
\end{proposition}

\begin{proof}
We have~$\fix_{T}(n)=\sum_{d\divides n}d^{a+1}=\sigma_{a+1}(d)$
and fixed points for iterates simply multiply for
Cartesian products so, for a prime~$p$
and~$k\ge1$, by~\eqref{andtheycamefromeverywhere},
\begin{eqnarray*}
\orbit_{T\times\cdots\times T}(p^k)&=&
\frac{1}{p^k}\sum_{d\divides p^k}\mu(p^k/d)(\sigma_{a+1}(d))^m\\
&=&\frac{1}{p^k}\(
\(\textstyle\frac{p^{(a+1)(k+1)}-1}{p^{a+1}-1}\)^m
-\(\textstyle\frac{p^{(a+1)k}-1}{p^{a+1}-1}\)^m
\).
\end{eqnarray*}
Clearly~$n\mapsto\orbit_{T\times\cdots\times T}(n)$ is
multiplicative, so this proves the proposition.
\end{proof}

Proposition~\ref{seekingaplacetostand} allows the orbit
Dirichlet series for higher powers to be computed (in
the case~$\dirichlet_T(s)=\zeta(s)$, with trivial
changes for~$\orbit_T(s)$ polynomial).
To this end, assume
that~$\dirichlet_T(s)=\zeta(s-a)$,
let~$f(n)=\orbit_{T\times\cdots\times T}(n)$,
write
\[
\dirichlet_{T\times\cdots\times T}(s)
=\prod_{p\in\mathcal P}\(1+f(p)p^{-s}+f(p^2)p^{-2s}+\cdots\)=
\prod_{p\in\mathcal P}E_p(s),
\]
and define~$\theta$ by
\[
f(p^k)=\frac{1}{(p^{a+1}-1)^{m-1}}\theta(p^k).
\]
Then
\[
E_p(s)=
1+\frac{1}{(p^{a+1}-1)^{m-1}}
\sum_{b=0}^{m-1}A_bp^{(b+1)a+b-s}
\frac{1}{1-p^{(b+1)a+b-s}},
\]
where~$A_b=(-1)^r\binom{m}{m-r}
\sum_{j=0}^{(b+1)a+b}p^j$, and~$m-r-1=b$, so by rearranging
\begin{eqnarray*}
E_p(s)&=&
\frac{M_p(s)}{
(1-p^{a-s})
(1-p^{2a+1-s})
(1-p^{3a+2-s})
\cdots(1-p^{ma+(m-1)-s})}
\end{eqnarray*}
with~$M_p(s)\neq0$.
Thus~$\dirichlet_{T\times\cdots\times T}(s)$
is given by
\[
\zeta(s-a)\zeta(s-(2a+1))
\zeta(s-(3a+2))
\cdots
\zeta(s-(ma+m-1))
\prod_{p\in\mathcal P}
M_p(s),
\]
where~$M_p(s)$ is (in principle) explicitly computable,
and so
\[
\sigma(\dirichlet_{T\times\cdots\times T})=ma+m.
\]

\begin{example}
By Perron's theorem~\cite{perron} we deduce that
if~$\dirichlet_T(s)=\zeta(s)$, then
\[
\pi_{T\times\cdots\times T}(N)
\sim C_m\zeta(m)\zeta(m-1)\cdots\zeta(2)\frac{N^m}{m}
\]
where~$C_m=\prod_{p}M_p(m)$ is an explicit constant.
Thus, for
example,~$\pi_T(N)\sim N$
and~$\pi_{T\times T}(N)\sim\frac{\pi^2}{12}N^2$,
while
\[
\pi_{T\times T\times T}(N)\sim
C_3\frac{\pi^2\zeta(3)}{18}N^3,
\]
where
\[
C_3=\prod_{p}(1+p^{-5}+2p^{-2}+2p^{-3})=2.835979\dots.
\]
\end{example}

\begin{example}
Example~\ref{heavyintheair} with~$a=b=0$
gives
\[
\dirichlet_{T\times T}(s)=\frac{\zeta(s)^2\zeta(s-1)}{\zeta(2s)},
\]
and the calculation above gives
\begin{equation}\label{oraplacetohide}
\dirichlet_{T\times T\times T}(s)
=
\zeta(s)\zeta(s-1)\zeta(s-2)\prod_{p\in\mathcal P}
\(1+(2p+2)p^{-s}+p^{1-2s}\).
\end{equation}
\end{example}

\begin{remark}
The Euler product~$\prod_{p}(1+p^{1-2s}+2p^{1-s}+2p^{-s})$
is suggestive, but deceptively so. Under the Hecke
correspondence, the modular form with Fourier series~$f(\tau)=
c(0)+\sum_{n=1}^{\infty}c(n){\rm e}^{2\pi{\rm i}n\tau}$
has associated Dirichlet series
\[
\phi(s)=
\sum_{n=1}^{\infty}\frac{c(n)}{n^s}=\prod_{p}
\(1-c(p)p^{-s}+p^{2k-1}p^{-2s}\)^{-1}.
\]
However,
there is no real connection because the choice
of parameters needed violates
the (known) Weil bounds that~$\vert r_2\vert=
\vert r_2\vert=\sqrt{p}$ where~$1-c(p)x+p^{2k-1}x^2=
(1-r_1x)(1-r_2x)$. (Equivalently, the
seeming relationship with an~$L$-function of an
elliptic curve is meaningless because the
choice of parameters violates the Hasse bounds).
\end{remark}

\section{Natural boundaries}

Natural boundaries for Dirichlet series arise in several contexts.
Esterman's theorem~\cite{estermann} gives a large class of Euler products
of the form
\[
\prod_{p}h(p^{-s})
\]
with
natural boundaries. The example below is more closely related to the
work of Grunewald, du Sautoy and Woodward~\cite{MR1878998},~\cite{MR2371185}
on zeta functions for subgroup
growth, where products of `ghost' polynomials are used to exhibit
natural boundaries for products of the form
\[
\prod_{p}h(p^{-s},p).
\]
Natural boundaries also arise for dynamical zeta functions in
several natural dynamical settings,
including certain
random maps~\cite{MR1925634}
and automorphisms of certain solenoids~\cite{MR2180241}.

We exhibit a natural boundary for a specific case, but
the appearance of a natural boundary for triple (and
higher) products of systems with polynomial orbit growth
is a widespread phenomena.

\begin{theorem}
If~$\dirichlet_T(s)=\zeta(s)$, then~$\dirichlet_{T\times T\times T}(s)$ has abscissa
of convergence at~$3$, a meromorphic extension to~$\Re(s)>1$, and a natural boundary at~$\Re(s)=1$.
\end{theorem}

\begin{proof}
By~\eqref{oraplacetohide} we know that
\[
\dirichlet_{T\times T\times T}(s)=\zeta(s)\zeta(s-1)\zeta(s-2)\prod_{p\in\mathbb P}f(p^{-s},p)
\]
where~$f(p^{-s},p)=(1+(2p+2)p^{-s}+p^{1-2s})$. The term~$\prod_{p}f(p^{-s},p)$ converges
for~$\Re(s)>2$, so the abscissa of convergence is determined by the term~$\zeta(s-2)$.

To show that~$\Re(s)=1$ is a natural boundary, we show that each point~$s$
with~$\Re(s)=1$ is a limit of a sequence~$(s_n)$ of zeros of~$\prod_{p}f(p^{-s},p)$
with~$\Re(s_n)>1$. Solving the quadratic~$f(x,p)=0$ for~$x$ gives
the solutions
\[
\alpha_p^{\pm}=-\(1+\textstyle\frac{1}{p}\)\pm\sqrt{\(1+
\textstyle\frac{1}{p}\)^2-\textstyle\frac{1}{p}},
\]
and the zero~$\alpha_p^{+}=p^{-s}$ has solutions
\[
s_{n,p}=\frac{-\log\vert\alpha_p^{+}\vert}{\log p}+\frac{\pi\imag+2k\pi\imag}{\log p}
\]
for~$k\in\mathbb Z$. Notice that~$\frac{-3+\sqrt{7}}{2}<\alpha_p^{+}<0$ for all~$p$,
$\alpha_p^+\to0$ as~$p\to\infty$, and
by the binomial theorem~$\alpha_p^{+}\sim-\frac{1}{2p}$
for large primes~$p$. It follows that~$\Re(s_{n,p})>1$ and~$\Re(s_{n,p})\to1$
as~$p\to\infty$. Thus given any~$s$ with~$\Re(s)=1$ we may choose a
sequence~$(s_{n_k,p_k})$ with the properties that
\begin{enumerate}
\item $\Re(s_{n_k,p_k})>1$;
\item $s_{n_k,p_k}\to s$ as~$k\to\infty$;
\item $\prod_{p}f(p^{-s_{n_k,p_k}},p)=0$ for all~$k\ge1$.
\end{enumerate}
This shows that~$\Re(s)=1$ is a natural boundary.

It remains to show that there is a meromorphic extension to
the half-plane~$\Re(s)>1$,
and here we follow the methods of~\cite{MR2371185}. Using the lexicographic
ordering on~$\mathbb N^2$ to eliminate terms in ascending powers of~$x$ and
then~$y$, there is a unique decomposition
\[
f(x,y)=\prod_{(m,n)\in\mathbb N^2}\(1-x^my^n\)^{c(m,n)}
\]
with~$c(m,n)\in\mathbb Z$,
where~$f(x,y)=1+2x+2xy+yx^2$. This may be constructed using factors of the shape~$(1-x^ay^b)^{e}$
to eliminate a term~$-ex^ay^b$ ($e>0$) and factors of the shape~$(1-x^{2a}y^{2b})^e(1-x^ay^b)^{-e}$
to eliminate a term~$ex^ay^b$ ($e>0$), obtaining an approximation valid to
larger and larger powers of~$x$ by induction. Thus, for example, we find
that
\[
f(x,y)=(1-x^2)^3(1-x)^{-2}(1-x^2y^2)^2(1-xy)^{-2}(1-x^2y)^{3}(1-x^2y^2)+O(x^3).
\]
By construction, if
\[
f(x,y)=\prod_{m\le M}\(1-x^my^n\)^{c(m,n)}+\sum_{m>M}e(m,n)x^my^n
\]
then if~$c(m,n)$ and~$e(m,n)$ are non-zero we must have~$n\le m$.
Thus for each~$M$ and~$\Re(s)>\max\{(n+1)/m\mid e(m,n)\neq0\}$ the
product
\[
f_M(s)=\prod_{p}\(
1+\frac{\sum_{m>M}e(m,n)p^{n-ms}}{\prod_{p}(1-p^{n-ms})^{c(m,n)}}
\)
\]
converges absolutely, allowing~$\prod_{p}f(p^{-s},p)$ to be defined
there by
\[
\prod_{(m,n)\in\mathbb N^2,m\le M}\zeta(ms-n)^{-c(m,n)}f_M(s).
\]
Letting~$M\to\infty$ gives a meromorphic extension to~$\Re(s)>1$.
\end{proof}


\end{document}